\documentclass[11pt]{amsart}
\usepackage{amscd,amssymb,float}

\newtheorem{theorem}{Theorem}
\newtheorem{lemma}[theorem]{Lemma}

\newtheorem{proposition}[theorem]{Proposition}

\renewcommand{\leq}{\leqslant}
\renewcommand{\geq}{\geqslant}

\theoremstyle{definition}

\theoremstyle{definition}

\numberwithin{equation}{section}

\numberwithin{equation}{section} \numberwithin{figure}{section}

\author{}
\address{}
\address{}
\email{}
\title{Remarks   on K\"ahler orbifolds of non-negative Ricci curvature}
\author{Yuguang Zhang}
\address{Xi'an Jiaotong-Liverpool University, Suzhou, China}
\email{Yuguang.Zhang@xjtlu.edu.cn}

\begin{document}
\begin{abstract}
 This note proves  orbifold versions of Kobayashi's theorem.  The main result  asserts that a compact K\"ahler orbifold with non-negative Ricci curvature, along with certain conditions regarding singularities, is simply connected.    
\end{abstract}

\maketitle
\footnotetext[1]{Supported in part by grant NSFC-12531001.}
\section{Introduction}
A theorem due to 
Kobayashi (cf. \cite{Ko}) asserts that a compact K\"ahler manifold with positive Ricci curvature is simply connected. This short note aims to generalise this theorem to K\"ahler orbifolds.  The concept of orbifolds was initially  introduced by Satake under the term `V-manifolds', and was  later  renamed  `orbifolds' by Thurston (cf. \cite{Sa,Th}). 

We briefly recall the definition of orbifolds and refer to \cite{Ba, BoZ, Ka1, KoT, Sa} for detailed definitions and  basic properties  of orbifolds.   A complex $n$-orbifold $X$ is a complex analytic space equipped with  a complex orbifold structure (See \cite{Ka2, Ba}). More precisely,  for any point $x\in X$, there exists a neighbourhood $U_x$  that is the quotient of a finite subgroup $G_x$ of $U(n)$ acting linearly   on an open neighbourhood $\tilde{U}_x$ of $0$ in $\mathbb{C}^{n}$, i.e., $\tilde{U}_x/G_x = U_x$. We  denote  the quotient map as  $q_x:\tilde{U}_x \rightarrow U_x$,  which satisfies $q_x(0)=x$, and call $U_x$ an  orbifold chart.  Furthermore, if $U_y \subseteq U_x$, then there is a holomorphic open   embedding $\iota: \tilde{U}_y \hookrightarrow \tilde{U}_x$ and an injective  homomorphism $\kappa: G_y  \rightarrow G_x$ such that $q_x \circ \iota =q_y $ and $\iota (t \cdot y')=\kappa(t)\iota ( y') $ for any $t\in G_y$ and $y'\in \tilde{U}_y$. 
 $X$ admits a regular/singular decomposition $X=X_R\coprod X_S$, where the regular part   $X_R$ is a complex  manifold of dimension $n$. The singular set $X_S$  consists  of those points with  non-trivial orbifold groups, i.e.,  $ x\in X_S$ if and only if $G_x\neq \{1\}$.   
   
   An example of complex orbifolds  is the quotient space $X=\mathbb{C}P^1/\mathbb{Z}_2$ where $\mathbb{Z}_2$ acts on $\mathbb{C}P^1=\mathbb{C} \cup \{\infty\}$ through the map  $z \mapsto -z$. Note that $X$ is  an orbifold  with $X_S$ consisting of 2 points,  even though  $X$ is homeomorphic to $\mathbb{C}P^1$ as topological spaces.  
  A complex projective surface $X$ with only finite ordinary double points as  singularities is a complex orbifold of dimension $2$, which is not homeomorphic to any manifold.

  A K\"ahler  metric $g$ (respectively  $(p,q)$-form $\beta$) on a complex  orbifold $X$ is a smooth K\"ahler   metric $g$ (resp. $(p,q)$-form $\beta$) on the regular part $X_R$, and on any  $q_x^{-1}(U_x \cap X_R)$, the pull-back  $\tilde{g} =q_x^* g$ (resp.  $\tilde{\beta} =q_x^* \beta$) extends to a smooth K\"ahler   metric  (resp. $(p,q)$-form) on $\tilde{U}_x\subseteq \mathbb{C}^{n}$.  We refer to  $(X, g)$ as  a K\"ahler orbifold.    
 
 Many classical theorems for manifolds have been shown to  hold for  orbifolds (\cite{Ba, Ba2, Bor, BoZ, Ka1, Ka2, Sa}), for example, the de Rham theorem, the Dolbeault theorem, and the Hirzebruch-Riemann-Roch formula, etc.   We aim to generalise Kobayashi's theorem to orbifolds in a manner that takes into account the effects of orbifold singularities on the results.  
 
 Note that $G_x$ acts on the $(1,0)$-cotangent space $T^{*(1,0)}_0 \tilde{U}_x$ at  $0\in \tilde{U}_x$, and thus also on  $\wedge^pT^{*(1,0)}_0 \tilde{U}_x$, for all $1\leq p \leq n$. Define  \begin{equation}\label{eq1}
   F_x^{p,0} =\{w\in \wedge^pT^{*(1,0)}_0 \tilde{U}_x| t \cdot w=w, \ {\rm for \ all} \ t\in G_x\},
  \end{equation} i.e., 
 the fixed point set of $G_x$-action on   $\wedge^pT^{*(1,0)}_0 \tilde{U}_x$.  If $x$ is a regular  point, i.e., $x\in X_R$, then $\dim_{\mathbb{C}} F_x^{1,0}=n $  and $\dim_{\mathbb{C}} F_x^{n,0}=1 $.  $x\in X_S$ if and only if $\dim_{\mathbb{C}} F_x^{1,0}<n $. If $x$ is an ADE surface  singularity, i.e., $ \mathbb{C}^2/G_0$ for a finite subgroup $ G_0$ of $SU(2)$,   then $\dim_{\mathbb{C}} F_x^{1,0}=0$ and $\dim_{\mathbb{C}} F_x^{2,0}=1$.
 
 The main result is the following theorem.

   \begin{theorem}\label{main-th}  Let $(X,g)$ be a  compact  K\"ahler  $n$-orbifold with non-negative Ricci curvature, i.e., ${\rm Ric}(g)\geq 0$.  If  \begin{equation}\label{eq2}\sum_{p=1}^{n}\inf_{x\in X}\dim_{\mathbb{C}} F_x^{p,0}=0,   \end{equation}  then  $X$ is simply connected, i.e., the fundamental group $\pi_1(X)$ is trivial.   
\end{theorem}

The hypothesis (\ref{eq2}), along with the assumption of non-negative Ricci curvature, plays a similar role to that of  positive Ricci curvature in Kobayashi's theorem.

 Now we apply Theorem \ref{main-th}  to an example. Let $T^2_{\mathbb{C}}=\mathbb{C}^2/(\mathbb{Z}^2+\sqrt{-1}\mathbb{Z}^2)$ be the complex 2-torus. If $z_1$ and $z_2$ are the angle coordinates on $T^2_{\mathbb{C}}$ induced from   coordinates on $\mathbb{C}^2$, then $\mathbb{Z}^2_2$ acts on $T^2_{\mathbb{C}} $ by $\gamma_1 \cdot (z_1, z_2)=(-z_1, -z_2)$ and $\gamma_2 \cdot (z_1, z_2)=(z_2, z_1)$, where $\gamma_1$ and $\gamma_2$ are generators of  $\mathbb{Z}^2_2$.  The quotient space $X=T^2_{\mathbb{C}}/\mathbb{Z}^2_2$ is a complex orbifold that admits a flat orbifold K\"ahler metric $g$ induced by the Euclidean  metric on $\mathbb{C}^2$.   Note that $(0,0)\in T^2_{\mathbb{C}}$ is fixed by the $\mathbb{Z}^2_2$-action. 
 We consider the orbifold  point $x_0$ in $X$ which   is the image of  $(0,0)\in T^2_{\mathbb{C}}$ under the quotient map.  The induced $\mathbb{Z}^2_2$-actions on $T^{*(1,0)}_{(0,0)} T^2_{\mathbb{C}}$ and $ \wedge^2 T^{*(1,0)}_{(0,0)} T^2_{\mathbb{C}}$ are  given by $\gamma_1 \cdot dz_i=-dz_i$, $i=1,2$, and $\gamma_2 \cdot dz_1\wedge dz_2=-dz_1\wedge dz_2$. Thus we have  $$\dim_{\mathbb{C}} F_{x_0}^{1,0}+\dim_{\mathbb{C}} F_{x_0}^{2,0}=0, $$ and $X$ is simply connected by Theorem \ref{main-th}.   We remark that although   $X$ is simply connected, the orbifold  fundamental group $\pi_1^{orb}(X)$ of  $X$ is not trivial, and the universal orbifold covering of $X$ is $\mathbb{C}^2$ (See \cite{Bor, BoZ} for the definitions of orbifold  fundamental group and of  universal orbifold covering). 

We also have a theorem that includes a broader range of orbifold types than those permitted in Theorem \ref{main-th}, such as the ADE surface singularities.

 \begin{theorem}\label{th-new}  Let $(X,g)$ be a  compact  K\"ahler  $n$-orbifold with non-negative Ricci curvature, i.e., ${\rm Ric}(g)\geq 0$.  If $n=2m$ and \begin{equation}\label{eq2+}\sum_{p=1}^{n-1}\inf_{x\in X}\dim_{\mathbb{C}} F_x^{p,0}=0,   \end{equation}  then  $X$ is either  simply connected  or the fundamental group $\pi_1(X)\cong \mathbb{Z}_2$.  Furthermore, 
  \begin{itemize}
\item[i)] 
 if $\pi_1(X)\cong \mathbb{Z}_2$, then    ${\rm Ric}(g)\equiv 0$, i.e., $(X,g)$ is a Calabi-Yau orbifold, 
 \item[ii)]  if $X$ admits a  non-vanishing  holomorphic $2m$-form $\Omega$, then $X$ is a   simply connected Calabi-Yau orbifold. 
  \end{itemize}
\end{theorem} 

An implication   of Theorem \ref{th-new} is that  the Kummer K3 orbifold $T^2_{\mathbb{C}}/\mathbb{Z}_2$, defined as the quotient of the complex two-dimensional torus $T^2_{\mathbb{C}}$ by the involution $(z_1, z_2) \mapsto (-z_1, -z_2)$, is simply connected.
  This fact has been proved in  \cite{Sp}.
   
   {\bf Acknowledgments. } The author thanks the referee for helpful   suggestions that have improved the present work.   

\section{Proofs}

 This section  proves 
 Theorem \ref{main-th} and  Theorem \ref{th-new}. We
 follow the  proof of Kobayashi's theorem, and  adapt  the argument  to the case of orbifolds.   
Let $X$ be a compact complex orbifold of dimension $n$,  $g$ be a K\"ahler metric on $X$, and  $\nabla$ be  the Levi-Civita connection of $g$.

 Denote  $A^{p,q}(X)$ (resp. $A^k(X)$) as the space of smooth $(p,q)$-forms (resp. $k$-forms) on $X$. As in the smooth case, the exterior differentiation $d: A^k(X) \rightarrow A^{k+1}(X) $,  the Cauchy-Riemann operator  $\bar{\partial}: A^{p,q}(X) \rightarrow A^{p,q+1}(X) $, and   $\partial: A^{p,q}(X) \rightarrow A^{p+1,q}(X) $ are well-defined, which satisfy  $d=\partial+ \bar{\partial}$ and $ d^2=\partial^2= \bar{\partial}^2=0$.   The De Rham cohomology $H^k(X)$ and the Dolbeault cohomology $H^{p,q}(X)$ are defined in the same way as in  the case of smooth manifolds (cf. \cite{Ba, Ba2, Ka2}).   Note that  $ H^{p,0}(X)$ consists of   holomorphic $p$-forms, i.e., $\beta \in A^{p,0}(X)$, $\bar{\partial} \beta=0$. 
  
  We also have the Laplace  operators $\Delta_d=(d+d^*)^2$, $\Delta_{\bar{\partial}}=(\bar{\partial}+\bar{\partial}^*)^2$, and $\Delta_{\partial}=(\partial+\partial^*)^2$,  where $d^*$, $\bar{\partial}^*$, and $\partial^*$ are the adjoint operators. 
 The same local calculation as in the manifold case shows $\Delta_d=2 \Delta_{\partial}=2 \Delta_{\bar{\partial}}$ (cf. Theorem 8.6 in \cite{Mo}). By the Hodge-Kodaira   decomposition theorem of  \cite{Ba2},  $H^{p,q}(X)$ (resp. $H^{k}(X)$) is isomorphic to the space $\mathcal{H}^{p,q}(X)$ (resp. $\mathcal{H}^{k}(X)$) of harmonic $(p,q)$-forms, i.e., $$H^{p,q}(X) \cong \mathcal{H}^{p,q}(X)=\{\beta \in A^{p,q}(X)| \Delta_{\bar{\partial}} \beta =0\}.  $$ Furthermore, $\mathcal{H}^{p,q}(X)= \overline{\mathcal{H}^{q, p}(X)} $ and $\dim_{\mathbb{C}} \mathcal{H}^{0,0}(X)=1 $.  The holomorphic Euler characteristic is defined as  $$\chi(X, \mathcal{O}_X)=\sum_{p=0}^{n}(-1)^p\dim_{\mathbb{C}}H^{0, p}(X)=\sum_{p=0}^{n}(-1)^p\dim_{\mathbb{C}}H^{ p, 0}(X).$$ 

 The following proposition  might have some independent interests.  
 
 \begin{proposition}\label{prop1} If  $(X,g)$ is  a  compact  K\"ahler  $n$-orbifold with non-negative Ricci curvature, i.e., ${\rm Ric}(g)\geq 0$, then for any $p\geq 1$,  $$\dim_{\mathbb{C}} H^{p, 0}(X) \leq \inf_{x\in X}\dim_{\mathbb{C}} F_x^{p,0}.    $$ 
\end{proposition} 

 \begin{proof} Note that 
 the Bochner technique applies  to  orbifolds (cf. \cite{CaT, Bor}).  
   The  Weitzenb\"ock  formula (cf.  \cite{Bo} and  Proposition 14.3 in \cite{Mo}) says that $$2\bar{\partial}^*\bar{\partial}\beta =2 \Delta_{\bar{\partial}} \beta =\nabla^*\nabla \beta + {\rm Ric}^p (\beta).$$ Here    ${\rm Ric}^p: A^{p,0} (X)\rightarrow A^{p,0} (X)$ is a zero order  linear  operator such that $g({\rm Ric}^p(\beta), \beta)\geq 0$, for any $\beta \in A^{p,0}(X)$,  when the Ricci curvature  of $g$ is non-negative, i.e., ${\rm Ric}(g)\geq 0$. 
  We recall that   for a function $f$ on $X$, the integration is defined as $$ \int_X f dv_g= \sum_{U_x \in \mathcal{U} } \frac{1}{|G_x|}\int_{\tilde{U}_x} (\eta_{U_x}  f)\circ q_x dv_{\tilde{g}},$$ (cf. \cite{Ba, Ka2})  where $|G_x|$ is the cardinality of $G_x$. Here 
 $ \mathcal{U}$ is a countable  collection of orbifold charts   $U_x$ that covers of $X$, and $\{\eta_{U_x}\}$ is a partition of unity  with respect to $ \mathcal{U}$, i.e.,  $\sum\limits_{U_x \in \mathcal{U} } \eta_{U_x} \equiv 1$.  
 Now, for any $\beta \in H^{p,0}(X)$, we  integrate  the inner product of $\beta$ with the Weitzenb\"ock  formula, and  obtain $$ 0=\int_X( |\nabla \beta|^2_g+g({\rm Ric}^p(\beta), \beta))dv_g\geq \int_X |\nabla \beta|^2_gdv_g\geq 0.$$ Thus $\nabla \beta \equiv 0$, i.e., all holomorphic $p$-forms $\beta$ are parallel.  

   If $\beta$ is a parallel holomorphic $p$-form, then for any $x\in X$, $q_x^*\beta$ extends to a parallel $G_x$-invariant  $p$-form $\tilde{\beta}$ on $\tilde{U}_x$.   This implies $\tilde{\beta}(0) \in  F_x^{p,0}$. By considering an orthonormal basis of $H^{p,0}(X)$, we obtain the  result.   
\end{proof}

Similar theorems have been obtained for Riemannian orbifolds (Theorem 2.2 of \cite{CaT} and Proposition 9.0.1 of \cite{Z}).  More precisely, the Bochner technique  can be deployed to show that the first Betti number  of a compact Riemannian orbifold $M$ with non-negative Ricci curvature is zero, i.e., $b_1(M)=0$,  if  $$\inf\limits_{x\in M}\dim_{\mathbb{R}} F_x^{1}=0,$$ where $F_x^{1}$ denotes the set of fixed points of the $G_x$-action on $T^*_0\tilde{U}_x$.  Furthermore,  the splitting theorem for orbifolds in \cite{BoZ} shows that  the fundamental group $\pi_1(M)$ is finite  in this  case. This result generalises Myers' theorem for Riemannian manifolds.  We only state and prove this fact  for K\"ahler orbifolds, as it will be used in the proof of the main theorems.

\begin{lemma}\label{le} Let   $(X,g)$ be  a  compact  K\"ahler  $n$-orbifold with  ${\rm Ric}(g)\geq 0$. If $$\inf_{x\in X}\dim_{\mathbb{C}} F_x^{1,0}=0,$$ then the fundamental group $\pi_1(X)$ is finite.  
\end{lemma} 

\begin{proof}
Let $\pi: \bar{X}\rightarrow X$ be the universal covering   of $X$ in the usual topological sense.  For an orbifold chart $U_x\subseteq X$,
 $\pi^{-1}(U_x)=\coprod\limits_{\nu\in \pi_1(X)}\bar{U}_x^\nu$, and 
 any connected  component $\bar{U}_x^\nu$  of $\pi^{-1}(U_x)$ is homeomorphic to $U_x$ by shrinking $U_x$ if necessary. Hence $\bar{U}_x^\nu=\tilde{U}_x/G_x $, and   $\bar{X}$  is equipped with  a natural orbifold structure induced by that of $X$, more precisely the collection of orbifold charts $\bar{\mathcal{U}}=\{\bar{U}_x^\nu| U_x\in \mathcal{U}, \pi^{-1}(U_x)=\coprod\limits_{\nu\in \pi_1(X)}\bar{U}_x^\nu\} $. 
 
 Theorem 1 in  \cite{BoZ} shows that $$\bar{X}=Y\times \mathbb{R}^l$$ where $Y$ is an orbifold containing no geodesic lines.  The same arguments as in the proof of Theorem 2 of \cite{BoZ} prove that $Y$ is compact by replacing the orbifold fundamental group in \cite{BoZ} with the usual topological  fundamental group. 

 For any $x\in X$, there is an orbifold chart $\bar{U}_x^\nu \subseteq \bar{X}$ homeomorphic to $U_x$, and  $\bar{U}_x^\nu=\tilde{U}_x/G_x $. Therefore, the fixed point set of the $G_x$-action on $\mathbb{C}^n$  contains at least a real linear subspace $\mathbb{R}^l$.    Since $G_x\subseteq U(n)$ and $T^*_0 \tilde{U}_x \otimes \mathbb{C}=T^{*(1,0)}_0 \tilde{U}_x \oplus T^{*(0,1)}_0 \tilde{U}_x  $, there is a non-trivial $w\in T^{*(1,0)}_0 \tilde{U}_x $ fixed by the $G_x$-action if $l>0$, i.e.,  Re$(w)\in T^*_0(\{0\} \times \mathbb{R}^l)\subseteq T^*_0 \tilde{U}_x $ and $t\cdot w=w$ for all $t\in G_x$.  This contradicts the hypotheses  $\inf\limits_{x\in X}\dim_{ \mathbb{C}}F^{1,0}_x=0$. Hence $l=0$, and the universal covering space $\bar{X}=Y$ is compact, which is equivalent to $\pi_1(X)$ being finite.   
\end{proof}

 The Hirzebruch-Riemann-Roch formula has been generalised to  complex orbifolds by Kawasaki in \cite{Ka2}. More precisely, $$\chi(X, \mathcal{O}_X)= \sum_{U_x \in \mathcal{U} } \frac{1}{|G_x|} \sum_{t \in G_x }\int_{\tilde{U}_x^t} \eta_{U_x}\circ q_x  {\rm Todd}^t(\tilde{g}),$$ 
 where  
 $ \mathcal{U}=\{U_x\}$  covers  $X$, and $\{\eta_{U_x}\}$ is a partition of unity  with respect to $ \mathcal{U}$, i.e.,  $\sum\limits_{U_x \in \mathcal{U} } \eta_{U_x} \equiv 1$.  Here $\tilde{g}=q_x^*g$,  $\tilde{U}_x^t=\{y \in \tilde{U}_x|t\cdot y=y\}$, for any $t\in G_x$, and   ${\rm Todd}^t(\tilde{g})$ is the equivariant Todd form on  $\tilde{U}_x^t$. Suppose that   the universal covering  $\pi: \bar{X}\rightarrow X$   of $X$ is compact, i.e., the fundamental group $\pi_1(X)$ is finite.
   Now the collection $\bar{\mathcal{U}}=\{\pi^{-1}(U_x)=\coprod\limits_{\nu\in \pi_1(X)}\bar{U}_x^\nu\}$ covers $\bar{X}$, and $\{\eta_{U_x}\circ \pi\}$ is a partition of unity on $\bar{X}$. We obtain  $$ \chi(\bar{X}, \mathcal{O}_{\bar{X}})= \sum_{\bar{U}_x^\nu \in \bar{\mathcal{U}} } \frac{1}{|G_x|} \sum_{t \in G_x }\int_{\tilde{U}_x^t} \eta_{U_x}\circ q_x  {\rm Todd}^t( \tilde{g})=|\pi_1(X)|\chi(X, \mathcal{O}_X).$$

\begin{proof}[Proof of Theorem \ref{main-th}]  Assume that the Ricci curvature ${\rm Ric}(g)\geq 0$, and   the condition  (\ref{eq2}) holds.  Hence we have $H^{p, 0}(X)=\{0\}$ for any $1\leq p\leq n$ by Proposition \ref{prop1}, and $$\chi(X, \mathcal{O}_X)=1.  $$ Lemma \ref{le} shows that $\pi_1(X)$ is finite.  
The Ricci curvature of $\pi^*g$ is  non-negative, and the condition of (\ref{eq2}) is also satisfied since $\bar{U}_x^\nu=\tilde{U}_x/G_x $ for any $\nu\in \pi_1(X) $, i.e., for each  $p=1, \cdots,  n$, 
 there is a point $y\in \bar{X}$ such that $\pi(y)=x$ and $$\dim_{\mathbb{C}} F_{y}^{p,0}=\dim_{\mathbb{C}} F_{x}^{p,0}=0.$$ Therefore $$ 1=\chi(\bar{X}, \mathcal{O}_{\bar{X}})=|\pi_1(X)|\chi(X, \mathcal{O}_X)= |\pi_1(X)|.$$ We conclude that $X$ is simply connected.   
 \end{proof}

 \begin{proof}[Proof of Theorem \ref{th-new}] The same arguments as above prove that the universal covering $\bar{X}$ is compact, i.e., $\pi_1(X)$ is finite.  
Since all holomorphic forms are parallel as shown in the proof of Proposition \ref{prop1}, we have $$\dim_{\mathbb{C}}H^{2m,0}(X)\leq 1, \  \  \  {\rm and }  \  \  \ \dim_{\mathbb{C}}H^{2m,0}(\bar{X})\leq 1.$$ By Proposition \ref{prop1},  the condition (\ref{eq2+}) implies that  $H^{p,0}(X)=H^{p,0}(\bar{X})=0$ for $1\leq p \leq 2m-1$ as above. Thus $\chi(\bar{X}, \mathcal{O}_{\bar{X}})$ and $\chi(X, \mathcal{O}_X)$ are either $1$ or $2$.  Note that  $\chi(\bar{X}, \mathcal{O}_{\bar{X}})\geq \chi(X, \mathcal{O}_X)$.   If $\chi(\bar{X}, \mathcal{O}_{\bar{X}})=\chi(X, \mathcal{O}_X)$, then $X$ is simply connected. If  $\chi(\bar{X}, \mathcal{O}_{\bar{X}})=2$ and $\chi(X, \mathcal{O}_X)=1$, then  $\pi_1(X)=\mathbb{Z}_2$.  In this case, $H^{2m,0}(\bar{X})$ is generated by a non-trivial  parallel holomorphic $2m$-form $\Omega$ on $ \bar{X}$. Hence, the local holonomy group of $\pi^*g$ is reduced to $SU(2m)$, which implies ${\rm Ric}(\pi^* g)=\pi^* {\rm Ric}(g)\equiv 0$ (cf. Chapter 10 in  \cite{Be}).  Finally, if there is a    non-vanishing  holomorphic $2m$-form $\Omega$ on $X$, then $\pi^*\Omega$ is also holomorphic on  $\bar{X}$ and  ${\rm Ric}( g)\equiv 0$. Therefore,   $ \chi(\bar{X}, \mathcal{O}_{\bar{X}})=\chi(X, \mathcal{O}_X)= 2, $ which implies that $\bar{X}=X $ and  $X$ is simply connected. 
 \end{proof}

\appendix
\section{ A corollary obtained by an AI}

The following result was discovered and proved by an artificial intelligence system (XIPU AI) using the preceding material. Only the theorem number has been modified by the author.

\begin{theorem}
Let $(X, g)$ be a compact K\"ahler $n$-dimensional orbifold with non-negative Ricci curvature, $\mathrm{Ric}(g) \ge 0$. If there exists an integer $1 \le k \le n$ such that
\[
\inf_{x \in X} \dim_\mathbb{C} F_{p,0}^{x} = 0 \quad \text{for all } 1 \le p \le k,
\]
then the fundamental group $\pi_1(X)$ is finite. Moreover, if $k = n$, then $X$ is simply connected.
\end{theorem}

\begin{proof}
By the Bochner technique (cf. Proposition 3), for each $1 \le p \le k$, all holomorphic $p$-forms vanish:
\[
H^{p,0}(X) = \{0\}.
\]

From Lemma 4, the condition $\inf_{x \in X} \dim_\mathbb{C} F_{1,0}^{x} = 0$ ensures that $\pi_1(X)$ is finite.

If $k = n$, then $H^{p,0}(X) = 0$ for all $1 \le p \le n$, which implies
\[
\chi(X, \mathcal{O}_X) = 1.
\]
For the universal cover $\bar X \to X$, we have
\[
\chi(\bar X, \mathcal{O}_{\bar X}) = |\pi_1(X)| \, \chi(X, \mathcal{O}_X),
\]
hence $|\pi_1(X)| = 1$, showing that $X$ is simply connected.
\end{proof}

\end{document}